\documentclass{amsart}

\usepackage{amsrefs}

\usepackage{amsmath,amsfonts,amssymb,amsthm}
\usepackage{mathrsfs}
\usepackage[svgnames]{xcolor}
\usepackage[colorlinks,linkcolor=blue,citecolor=blue,urlcolor=Navy,breaklinks]{hyperref}

\usepackage{comment}
\usepackage{graphicx}
\usepackage{caption,subcaption}

\newcommand{\Q}{{\mathbb Q}}
\newcommand{\Z}{{\mathbb Z}}

\newcommand{\R}{{\mathbb R}}
\newcommand{\F}{{\mathbb F}}

\newcommand{\X}{\mathcal{X}}

\newcommand{\cC}{\mathcal{C}}

\newcommand{\cH}{\mathcal{H}}

\newcommand{\cJ}{\mathcal{J}}

\newcommand{\cV}{\mathcal{V}}

\newcommand{\cU}{\mathcal{U}}

\newcommand{\Cu}{\mathfrak{C}}

\newcommand{\A}{\mathbb{A}}

\newcommand{\Jac}{\mathrm{Jac}}

\newcommand{\disc}{\mathrm{disc}}
\newcommand{\Ht}{H}

\newcommand{\ua}{\underline{a}}

\DeclareMathOperator{\ord}{ord}

\DeclareMathOperator{\Spec}{\mathrm{Spec}}

\DeclareMathOperator{\rank}{\mathrm{rank}}
\DeclareMathOperator{\glue}{\mathrm{glue}}

\newtheorem{definition}{Definition}
\newtheorem{theorem}{Theorem}
\newtheorem{problem}{Problem}
\newtheorem{proposition}{Proposition}[section]
\newtheorem{lemma}[proposition]{Lemma}
\newtheorem{remark}[proposition]{Remark}
\newtheorem{corollary}[proposition]{Corollary}

\title[Logarithmic Density of Positive-Rank Genus-2 Jacobians]{Logarithmic Density of Rank $\geq 1$ and Rank $\geq 2$ Genus-2 Jacobians and Applications to Hyperelliptic Curve Cryptography}

\author[R. Barbulescu]{Razvan Barbulescu}
\address{Institut de mathématiques de Bordeaux (Bordeaux INP, CNRS, Univ. Bordeaux) and Inria Bordeaux}
\email{razvan.barbulescu@u-bordeaux.fr}
\author[M. Barcau \and V. Pa\c sol]{Mugurel Barcau \and
Vicen\c tiu Pa\c sol}
\address{Institute of Mathematics of the Romanian Academy \and CertSIGN, Bucharest}
\email{mugurel.barcau@imar.ro; vicentiu.pasol@imar.ro}
\author[G. C. \c Turca\c s]{George C. \c Turca\c s}
\address{Babe\c s-Bolyai University, Cluj-Napoca \and certSIGN, Bucharest}
\email{george.turcas@ubbcluj.ro}

\thanks{This work was supported by the project “Group schemes, root systems, and
related representations” funded by the European Union - NextGenerationEU through
Romania’s National Recovery and Resilience Plan (PNRR) call no. PNRR-III-C9-2023-
I8, Project CF159/31.07.2023, and coordinated by the Ministry of Research, Innovation
and Digitalization (MCID) of Romania.
The ﬁrst author was supported by Agence Nationale de la Recherche under grant
ANR-22-PNCQ-0002.}

\begin{document}

\begin{abstract} In this work we study quantitative existence results for genus-$2$ curves over $\Q$ whose Jacobians have Mordell--Weil rank at least $1$ or $2$, ordering the curves by the naive height of their integral Weierstrass models. We use geometric techniques to show that asymptotically the Jacobians of almost all integral models with two rational points at infinity have rank $r \geq 1$. Since there are $\asymp X^{\frac{13}{2}}$ such models among the $X^7$ curves $y^2=f(x)$ of height at most $X$, this yields a lower bound of logarithmic density $13/14$ for the subset of such curves whose Jacobians have rank at least $1$. We further present a large explicit subfamily of genus-$2$ curves, ordered by height as above, for which the Jacobians have rank $r \geq 2$, yielding an unconditional logarithmic density of at least $5/7$.  Independently, we give a construction of genus-$2$ curves with split Jacobian and rank at least $2$, producing a subfamily of logarithmic density at least $2/21$. Finally, we analyze quadratic and biquadratic twist families in the split-Jacobian setting, obtaining a positive proportion of rank-$2$ twists. These results have implications for Regev's quantum algorithm in hyperelliptic curve cryptography.
 \end{abstract}

\maketitle 

\section{Introduction}

The Mordell--Weil theorem implies that for any smooth projective curve $C$ of genus $g \geq 1$ defined over $\Q$, the group of rational points on its Jacobian admits a decomposition $\Jac(C)(\Q) \simeq T \times \mathbb Z^{r}$, where $T$ is a finite group and $r \geq 0$ is an integer called the (Mordell--Weil) rank of $\Jac(C)(\Q)$. A fundamental problem in arithmetic geometry is to understand how these ranks are distributed in families of curves, and in particular among hyperelliptic curves of a fixed genus $g$.

 To ask precise distribution questions and give quantitative answers, one must first choose an ordering of the family, a notion of arithmetic \textit{size} $H(C)$ and then count curves $C$ with $H(C) \leq X$ as $X \to \infty $. Common choices include ordering by conductor, discriminant, or by various height functions on coefficients of integral models. Since counting by conductor is  notoriously difficult (see \cite{Brumer1992}), here we  enumerate integral models by the height of the coefficients of their models. 

 \medskip

In particular, for a model of the form $C: y^2=\sum_{i=0}^{2g+2}f_ix^i$ with $f_i \in \mathbb Z$, we define the height by $H_1(C) = \max\{|f_i|\}$ as in~\cite{BhargavaGross2017}.  This gives rise to the following box
 \begin{equation}\label{eq:cC1}
 \cC_1^{(g)}(X)=\left\{C: y^2=\sum_{i=0}^{2g+2}f_ix^i\in \Z[x] \middle| H_1(C) \leq  X \right\} \end{equation}
of cardinality  $(\gamma_1+o(1))  X^{2g+3}$, for an explicit constant $\gamma_1>0$. This is the ordering used by Booker et al. in \cite{Booker2016}.

\medskip

Another height function, considered in \cite{BhargavaGross2013} for models of the form $C: y^2=x^{2g+1}+\sum_{k=2}^{2g+1}c_kx^{2g+1-k}$ with $c_i \in \mathbb Z$, is given by $H_2(C) = \max\{|c_k|^{\frac{2g(2g+1)}{k}}\}$, which gives rise to the box
\begin{equation}\label{eq:cC2}
\cC_2^{(g)}(X)=\left\{C : y^2=x^{2g+1}+\sum_{k=2}^{2g+1}c_kx^{2g+1-k}\in \Z[x] \middle|  H_2(C) \leq X\right\} . \end{equation}

By \cite{BhargavaGross2013}*{Theorem 44} the cardinality of $\cC_2^{(g)}(X)$ is $(\gamma_2+o(1)) X^{\frac{2g+3}{4g+2}}$ for an explicit constant~$\gamma_2>0$. 

Let $\cC^{(g)}(X)$ denote one of the families of genus-$g$ curves defined above (i.e., $\cC_1^{(g)}(X)$ or $\cC_2^{(g)}(X)$).
If $S$ is any set of hyperelliptic curves, we define $S(X):=S\bigcap \cC^{(g)}(X)$. We say that $S$ has  \emph{proportion}  $\delta\in [0,1]$  (with respect to the family $\cC^{(g)}(X)$) if $\lim_{X\rightarrow\infty} \frac{|S(X)|}{|\cC^{(g)}(X)|}=\delta,$ whenever the limit exists. In many arithmetic situations of interest, such as hyperelliptic curves whose Jacobian has large Mordell--Weil rank, one expects this proportion to be zero. To capture the asymptotic size of such thin subsets, we introduce a logarithmic notion of density.

\begin{definition}\label{def:density}
A set $S$ of hyperelliptic curves of genus $g$ has logarithmic density $\alpha$ with respect to $\cC^{(g)}(X)$ if  
$$\liminf_{X\rightarrow\infty}\frac{\log |S(X)|}{\log|\cC^{(g)}(X)|}=\alpha. $$
\end{definition}

In what follows, all notions of proportion and density are taken with respect to the family $\cC_1^{(g)}(X)$, unless otherwise specified.

The main results of this article, some summarized in the next theorem, focus on the genus-$2$ case. Accordingly,
when referring specifically to genus-$2$ curves, we omit the superscript
$(g)$ and write $\cC(X)$ in place of $\cC^{(g)}(X)$.

\medskip

\begin{theorem}[see Corollaries \ref{cor:3}, \ref{cor:rank2-logdensity} and Prop.~\ref{prop:split}]\label{th:1}~

\begin{enumerate}
\item The subset of genus-$2$ curves with split Jacobian and rank $r\geq 2$ has logarithmic density at least $\frac{2}{21}$ with respect to $\cC_1(X)$.
\item The subset of genus-$2$ curves with rank $r\geq 1$ has logarithmic density at least $\frac{13}{14}$ with respect to $\cC_1(X)$.
\item The subset of genus-$2$ curves with rank $r\geq 2$ has logarithmic density at least $\frac{5}{7}$ with respect to $\cC_1(X)$.
\end{enumerate}
\end{theorem}
For comparison, in the case of elliptic curves the random matrix model predicts
that the set of curves of rank \(r\geq 2\) has logarithmic density
\(\frac{19}{24}\) with respect to \(\cC_2^{(1)}(X)\) (see
Theorem 7.3.3 in~\cite{Park2019}). In genus \(2\), we prove unconditional lower
bounds \(\frac{13}{14}\) for rank \(r \geq 1\) and \(5/7\) for rank \(r \geq 2\)
with respect to \(\mathcal{C}_1(X)\). If rank parities are equidistributed in the
subfamily with two points at infinity, then the same construction would suggest
logarithmic density at least \(\frac{13}{14}>\frac{19}{24}\) for rank \(r \geq 2\)
as well. This phenomenon is also consistent with existing numerical data: for
\(X \leq 2^{18}\), the percentage of curves of rank at least \(2\) is larger
among the Jacobians of genus-\(2\) curves than among elliptic curves
(see~\cite{OnlineComplement}).

Our work toward Theorem~\ref{th:1} is partially motivated by
\cite{Booker2016}*{Remark A.1}, where the authors observe that, for genus-\(2\)
curves of small discriminant with small coefficients, the divisor class
\([\infty_+-\infty_-]\) is rational and typically has infinite order. Indeed,
small coefficients make it more likely that the leading coefficient of
\(g(x)=4f(x)+h(x)^2\) is a square, which implies rationality of the two points at
infinity in the sextic case. We turn this heuristic into a conceptual and
quantitative statement by viewing \([\infty_+-\infty_-]\) as a universal section
of the universal Jacobian, proving that this section is not torsion, and showing
that among the integral models for which the two points at infinity are rational,
the proportion for which \([\infty_+-\infty_-]\) is torsion tends to \(0\). We
also produce in Corollary~\ref{cor:rank2-logdensity} an unconditional source of
Jacobian rank at least \(2\) by constructing an explicit two-section subfamily
and applying N\'eron specialization. This offers a partial explanation for the
abundance of rank-\(2\) examples in small-coefficient or small-discriminant
databases.

\begin{remark}[Brute-force search complexity] \label{rmk:brute-force} A naive brute-force search that draws curves uniformly at random from $\cC_{1}(X)$ would succeed in finding a curve of rank $r \geq 1$ in time $|\cC_1(X)|^{1-\alpha + o(1)}$, where $\alpha$ is the logarithmic density of such curves in $\cC_1(X)$. A consequence of the second part of this theorem is that the complexity of this algorithm is at most $O(X^{\frac{1}{2}+o(1)})$. Similarly, if such an algorithm is used to find curves of rank $r \geq 2$ by sampling uniformly at random from $\cC_1(X)$, the third part of the theorem implies that the algorithm succeeds after at most $O(X^{2+o(1)})$ trials.
\end{remark}

\medskip

\paragraph{\emph{The twists of a given curve.}} Let $C : y^2 = f(x)$ be a fixed hyperelliptic curve defined over $\Q$. For each
squarefree integer $d$, we denote by $C^{(d)}$ the quadratic twist of $C$ given by $dy^2=f(x)$. We now study rank distributions within the family of quadratic twists of this
fixed hyperelliptic curve. We
consider the set
\begin{equation}\label{eq:twist-DX}
D(X)=\{d \in \mathbb Z \mid 1\leq d<X\text{ and } d \text{ is squarefree}\}.
\end{equation}

In complete analogy with a definition above, for any set $S$ of squarefree integers, we set
$S(X) := S \cap D(X)$. We say that $S$ has \emph{proportion} $\delta \in \R_{\ge 0} \cup \{\infty\}$ if
the limit $\lim\limits_{X \to \infty} \frac{|S(X)|}{|D(X)|} = \delta$
exists. To measure the asymptotic size of subsets that have proportion $0$ within
a fixed twist family, we introduce the following logarithmic notion of density.

\begin{definition}
Let $C$ be a hyperelliptic curve and $r \ge 0$ be an integer. The logarithmic density of the set of twists $C^{(d)}$ with $\rank(\Jac(C^{(d)})) \ge r$ is defined by

\[
\alpha(C, r):=\liminf_{X \to \infty}
\frac{
\log \bigl| \{ d \in D(X) \mid \rank(\Jac(C^{(d)})(\Q)) \ge r \} \bigr|
}{
\log |D(X)|
}.
\]
 
\end{definition}

For example, in~\cite{Smith2025} it is proven under BSD that, for all elliptic curves $E$, the proportion of rank $r\geq 2$ twists of $E$ is~$0$. The proportion of odd rank twists exists for all hyperelliptic curves (see~\cite{Yu2019}), but it can be zero. It has been conjectured in~\cite{Watkins2008} that it has a logarithmic density $\alpha=\frac{3}{4}$ whereas the best proven result is that of \cite{Gouvea1991}, where the authors show that $\alpha\geq \frac{1}{2}$. \medskip

We prove a result on the twists of a genus-$2$ curve with split Jacobian.

\begin{theorem}[simplified statement of Prop.~\ref{split twist} and Prop.~\ref{square twist}]\label{th:2}~\\

\begin{enumerate}
\item Let \(E_1,E_2/\mathbb Q\) be elliptic curves with full rational
\(2\)-torsion, together with an ordering of their nonzero \(2\)-torsion points.
Assume that isomorphism \(E_1[2]\simeq E_2[2]\) is not induced by
a \(\overline{\mathbb Q}\)-isomorphism \(E_1\simeq E_2\).
For squarefree integers \(d_1,d_2\), let \(C^{(d_1,d_2)}\) be the genus-\(2\)
curve obtained by gluing the quadratic twists \(E_1^{(d_1)}\) and
\(E_2^{(d_2)}\) with the induced ordering of their \(2\)-torsion. 
If the Birch--Swinnerton-Dyer conjecture holds for the quadratic twists of
\(E_1\) and \(E_2\), then
\[
\lim_{X\to\infty}
\frac{
\left|\{(d_1,d_2)\in D(X)^2:
\operatorname{rank}\Jac(C^{(d_1,d_2)})(\mathbb Q)\ge 2\}\right|
}{
|D(X)|^2
}
=
\frac14 .
\]

\item Let \(E/\mathbb Q\) be a semistable elliptic curve admitting a
\(\mathbb Q\)-rational \(3\)-isogeny, and let \(C/\mathbb Q\) be a genus-\(2\)
curve such that \(\Jac(C)\sim_{\mathbb Q}E^2\). Then, for every nonzero
squarefree integer \(d\),
\[
\Jac(C^{(d)})\sim_{\mathbb Q}\bigl(E^{(d)}\bigr)^2,
\]
and
\[
\liminf_{X\to\infty}
\frac{
\#\{d\in D(X): \rank_{\mathbb Q}\Jac(C^{(d)})\ge 2\}
}{
|D(X)|
}
>0.
\]
\end{enumerate}
\end{theorem}

The theoretical arguments are accompanied by explicit computational experiments;
the scripts illustrating the results are publicly available in the GitHub
repository \cite{OnlineComplement}.

\subsection{Cryptographic consequences}\label{ssec:crypto}

Hyperelliptic curve cryptography (HECC) is the genus-$2$ analogue of elliptic
curve cryptography (ECC): one works in the Jacobian group
$\Jac(C)(\F_q)$ of a genus-$2$ curve $C/\F_q$. In the classical proposal of
K\"oblitz~\cite{Koblitz1989}, the curve is sampled from a large family of models
$C:y^2+h(x)y=f(x)$ over $\F_q$, so that security is tied to the average-case
difficulty of the discrete logarithm problem in these Jacobian groups. For
simplicity in this cryptographic discussion, we take \(q\) to be prime.

Shor's algorithm~\cite{Shor1994} already gives a polynomial-time quantum
algorithm for discrete logarithms. The relevance of Regev's work is different:
Regev's factoring algorithm~\cite{Regev2023}, and its extensions to the discrete
logarithm problem by Eker{\aa} and G\"artner~\cite{Ekera2024}, provide a new
quantum approach whose cost per run, i.e. the size of the quantum circuit, depends on an additional parameter. This approach was
adapted to elliptic curves in~\cite{BarbulescuBarcauPasol2025} and to genus-$2$
Jacobians in~\cite{BarbulescuBisson2024}.

In the curve setting, the Regev parameter can be supplied by rational
Mordell--Weil generators on a lift or twist of the curve. Section~\ref{ssec:Regev}
makes this precise: subject to the usual independence heuristic for the reduced
Mordell--Weil points, a rank-\(r\) lift or compatible twist allows one to take
the Regev parameter to be \(d=r+2\). This is the sense in which high
Mordell--Weil rank is useful for the Regev-style cryptanalytic preprocessing
step. Throughout the paper, when we say that a curve is suitable for Regev's
algorithm, we mean suitable as an input to this attack or experimental
preprocessing procedure, not suitable as a recommendation for secure HECC curve
selection.

This leads first to the corresponding search problem for rational curves.
\begin{problem}[THE HIGH RANK CURVE PROBLEM]\label{prob:draw}
Fix an integer \(g\ge 1\), a target rank \(r\ge 0\), and a height-ordered family
\(\mathcal F^{(g)}\) of genus-\(g\) hyperelliptic curves over \(\mathbb Q\), for
instance \(\mathcal C^{(g)}_1\), \(\mathcal C^{(g)}_2\), or a prescribed
height-ordered subfamily such as the split-Jacobian locus. For \(X\ge 1\), write
\(\mathcal F^{(g)}(X)\) for the curves in this family of height at most \(X\).
Given \(X\), sample a curve from the finite set
\[
\left\{
C\in \mathcal F^{(g)}(X) :
\operatorname{rank}\Jac(C)(\mathbb Q)\ge r
\right\},
\]
if this set is nonempty.
\end{problem}

Theorem~\ref{th:1} gives unconditional logarithmic-density lower bounds for this
search in our height ordering, and Remark~\ref{rmk:brute-force} translates these
bounds into upper bounds on the number of random trials needed to find curves of
rank at least \(1\) or at least \(2\). The explicit high-rank curves collected in
Appendix~\ref{appendix} serve a different purpose: they provide concrete inputs
for experiments with the Regev-style algorithm, rather than evidence for the
density theorems.

For cryptographic applications, however, the curve over \(\F_q\) is often fixed
in advance. In that case one cannot freely sample a new rational curve; one
instead searches for a high-rank lift, twist, or related curve whose reduction is
isomorphic to the original curve, or whose reduced Jacobian is isogenous to the
original Jacobian with suitable kernel. This leads to the second search problem.
\begin{problem}[THE HIGH RANK TWIST PROBLEM]\label{prob:twist}
Let \(C\) be a rational genus-\(2\) curve and let \(r\) be a fixed nonnegative
integer.

\begin{enumerate}
\item If \(\Jac(C)\) is simple, find, if it exists, a squarefree integer
\(d\in[1,X]\) such that
\[
\operatorname{rank}\Jac(C^{(d)})(\mathbb Q)=r.
\]
\item If \(\Jac(C)\sim_{\mathbb Q}E^2\) for an elliptic curve \(E/\mathbb Q\),
find, if it exists, a squarefree integer \(d\in[1,X]\) such that
\[
2\operatorname{rank}\bigl((E^{(d)})(\mathbb Q)\bigr)=r,
\]
equivalently, under the induced isogeny,
\[
\operatorname{rank}\Jac(C^{(d)})(\mathbb Q)=r.
\]
\item Let \(E_1/\mathbb Q\) and \(E_2/\mathbb Q\) be elliptic curves satisfying
the same conditions as in part (1) of Theorem~\ref{th:2} above. Find, if they exist, squarefree
integers \((d_1,d_2)\in[1,X]^2\) such that \(C^{(d_1,d_2)}\) satisfies
\[
\operatorname{rank}\Jac(C^{(d_1,d_2)})(\mathbb Q)
=
\operatorname{rank}E_1^{(d_1)}(\mathbb Q)
+
\operatorname{rank}E_2^{(d_2)}(\mathbb Q)
=r,
\]
where \(C^{(d_1,d_2)}\) is constructed as above.
\end{enumerate}
\end{problem}

The Poincar\'e complete reducibility theorem motivates the cases separated in
Problem~\ref{prob:twist}: after extending the base field, the Jacobian of a
genus-\(2\) curve is either simple, isogenous to a product \(E_1\times E_2\) of
two non-isogenous elliptic curves, or isogenous to \(E^2\) for an elliptic curve
\(E\). The split cases are especially amenable to analysis because the genus-\(2\)
rank problem reduces to rank questions for elliptic curves; the compatibility of
twisting with these split constructions is made precise in Lemma~\ref{gluelemma}
and Propositions~\ref{square twist} and~\ref{split twist}. Theorem~\ref{th:2},
under its stated hypotheses, asserts that the relevant split-Jacobian twist
families contain a positive proportion of twists of Jacobian rank at least \(2\).

Several structured subclasses of genus-$2$ curves used in cryptography, including
split-Jacobian and CM families, have been studied extensively
\cites{Buhler1998,Brown2005,Drylo2012,Kawazoe2008}. We are not aware of a
discrete-logarithm algorithm prior to this work which is faster for these
subclasses than for general genus-$2$ Jacobians. The distinction in the present
paper is that these structures can make the high-rank preprocessing used by the
Regev-style algorithm easier. For CM curves of the form $y^2=x^5+a$, one can
exploit dedicated arithmetic methods such as those in~\cite{Stoll2002} and
\cite{Molin2022}; for split Jacobians, the constructions of
Section~\ref{ssec:twist of split curve} reduce the search to quadratic twists of
elliptic curves. The computations in Section~\ref{ssec:twist of split} illustrate
these two phenomena on curves appearing in the cryptographic literature.

\subsection*{Acknowledgments} We would like to thank Bill Allombert and Aurel Page for discussions about the implementation of \(L\)-functions of genus-\(2\) curves and Hecke characters in PARI/GP. We also thank Alex Bartel for suggesting, in a personal discussion, that the rank distribution of twists of certain curves used in cryptography may be explained by the decomposition of their Jacobians. Finally, we are very grateful to the anonymous referees for their careful reading of the manuscript and for their detailed, thoughtful, and constructive reports, which led us to correct several mathematical points, clarify the exposition, and improve the organization of the paper.

\section{Main results}

\subsection{Number of curves whose Jacobian has positive rank}

\label{ssec:torsion}

Following \cite{Booker2016}*{Section 2.1}, every genus-$2$ curve defined over an integral domain of characteristic
$\neq 2$ admits a Weierstrass model
\begin{equation}\label{eq:booker-weierstrass}
  C:\quad y^2+h(x)y=f(x),
\end{equation}
with $\deg f\le 6$ and $\deg h\le 3$.

Completing the square gives the simplified equation
\begin{equation}\label{eq:booker-square}
  (2y+h(x))^2 \;=\; 4f(x)+h(x)^2 \;=:\; g(x).
\end{equation}
The discriminant of \eqref{eq:booker-weierstrass} is defined by
\[
  \Delta(f,h):=2^{-12}\disc\bigl(4f+h^2\bigr),
\]
and $\Delta(f,h)\neq 0$ is equivalent to \eqref{eq:booker-weierstrass} defining a smooth
genus-$2$ curve.

In the search for genus-$2$ curves of bounded discriminant, Booker et al. \cite{Booker2016}*{Section 3.1} enumerate integral
models of the form \eqref{eq:booker-weierstrass}. Although these models are more general than the ones in the boxes $\cC_1(X)$ and $\cC_2(X)$ (see \eqref{eq:cC1} and \eqref{eq:cC2}), one may reduce
\(h\) modulo \(2\) without changing the isomorphism class of the curve. Indeed,
for any \(q(x)\in \mathbb Z[x]\) with \(\deg q\le 3\), the change of variables
\(Y=y+q(x)\) transforms $y^2+h(x)y=f(x)$
into the isomorphic integral model
\[
   Y^2+(h(x)-2q(x))Y=f(x)+q(x)h(x)-q(x)^2.
\]
Moreover,
$
   4\bigl(f+qh-q^2\bigr)+(h-2q)^2=4f+h^2$,
so the completed-square polynomial, and hence the discriminant, is unchanged.
Choosing \(q\) coefficientwise therefore allows us to take every coefficient of
\(h\) to be \(0\) or \(1\), which explains the condition \(h_i\in\{0,1\}\) in the
boxes below.

For parameters $X,Y\ge 1$  the authors of \cite{Booker2016} define the boxes
\[\left\{
\begin{array}{l}
  S_1(X):=\{(f,h): |a_i|\le X,\ h_i\in\{0,1\}\}\\ 
  S_2(X,Y):=\{(f,h): |a_i|\le X\,Y^{6-i},\ h_i\in\{0,1\}\}.
\end{array}\right.,
\]
where $f=\sum_{i=0}^6a_ix^i$ and $h=\sum_{i=0}^{\deg h}h_ix^i$.
Thus \(S_1(X)\) decomposes as the disjoint union of the \(16\) fibers obtained by
fixing \(h\in\{0,1\}^4\); each fiber has the same coefficient bounds for \(f\) as
\(\mathcal C_1(X)\), but parametrizes the models
\[
   y^2+h(x)y=f(x)
\]
with that fixed choice of \(h\).

This motivates the definition of the following parameter space $\cU^{\infty}$.

\paragraph{\emph{The parameter space $\cU^{\infty}$.}}
Let $\cH$ be the reduced finite $\Q$-scheme with $\cH(\Q)=\{0,1\}^4$ parametrizing the $16$ possible
choices of $h$ with coefficients in $\{0,1\}$.  Let
\[
  \X \;:=\; \mathbb A^7_\Q \times \cH
\]
be the parameter space of pairs $(f,h)$ with $\deg f\le 6$, $\deg h\le 3$, and $h_i\in\{0,1\}$.
We use coordinates $\underline{a}=(a_6,\dots,a_0)$ on the $\A^7$-factor, so that
\[
  f_{\ua}(x)=a_6x^6+a_5x^5+\cdots+a_0 \text{ and }
  h(x)=h_3x^3+h_2x^2+h_1x+h_0
\]
with $h_i \in \{0,1 \}$.
Since $\X$ is a disjoint union of $16$ copies of $\A^7_\Q$, all arguments below should be read after fixing one $h$; in particular, for any irreducible component $\X_h\simeq \A^7_\Q$ and we will simply write $\X$. 

Consider the universal genus-$2$ curve $\Cu \to \mathcal X$,
defined as the projective closure of the affine equation
\begin{equation} \label{affine:model} 
y^2 + h(x)y = f_{\ua}(x) 
\end{equation}
inside the weighted projective space $\mathbb P(1,3,1)$.
 Let $\mathcal U \subset \mathcal X$ denote the open subset where the discriminant of this equation is non-vanishing. Over \(\mathcal U\), the morphism
\(\mathcal C \to \mathcal U\) is smooth and proper of relative dimension \(1\),
and its geometric fibers are curves of genus \(2\). In addition, the Jacobian of
\(\mathcal C\) is an abelian scheme \(\mathcal J\) over \(\mathcal U\).

\medskip

\paragraph{\emph{The two points at infinity and the universal section.}}
Write $g(x):=4f_{\ua}(x)+h(x)^2$ and set
\[
  c:=\mathrm{coeff}_{x^6}(g)=4a_6+h_3^2 \in \Gamma(\cU,\mathcal O_{\cU}).
\]
Let $\cU_6\subset \cU$ be the open subset where $c\neq 0$ (equivalently, $\deg g=6$).
In general the two points at $\infty$ are only defined after adjoining a square root of $c$,
so we form the finite étale double cover
  $\pi:\cU^\infty \longrightarrow \cU_6$,
where $\cU^{\infty}:= \Spec\bigl(\mathcal O_{\cU_6}[u]/(u^2-c)\bigr)$.
On the weighted projective closure of \eqref{affine:model} over $\cU^\infty$ the fiber at infinity
consists of two sections $\infty_\pm=(1:\pm u:0)$

We then define
\[
  \alpha_{\mathrm{univ}} := [\infty_+ - \infty_-] \in \mathcal J(\cU^{\infty}).
\]
For any rational point $\ua \in \cU^\infty(\Q)$ we write $\Cu_{\ua}$, $\cJ_{\ua}$ and $\alpha_{\ua}$ for the corresponding
specializations.

We now study the torsion loci of this universal section.

\medskip

\paragraph{\emph{Algebraicity of torsion loci.}}

For every integer $n\ge 1$ we define the \emph{$n$-torsion locus} of $\alpha_{\mathrm{univ}}$ by
\[
  \cV_n := \{ \ua \in \cU^\infty : n\alpha_{\mathrm{univ}}(\ua )=0\in \cJ_{\ua}\,\}.
\]

\begin{proposition}\label{prop:algebraic-booker}
For every positive integer $n$, $\cV_n$ is a Zariski-closed subset of $\cU^\infty$ and is defined over $\Q$.
\end{proposition}

\begin{proof}
Recall $\cJ\to\cU^\infty$ is an abelian scheme.  Let
$
  s:\cU^\infty \rightarrow \cJ
$
be the section corresponding to $\alpha_{\mathrm{univ}}$, $[n]:\cJ\to \cJ$ the
multiplication-by-$n$ morphism, and write
$
  e:\cU^\infty\rightarrow\mathcal J
$
for the zero section.

By definition, $\ua \in \cU^\infty$ lies in $\cV_n$ if and only if $[n](s(\ua))=e(\ua)$ in the fiber $\mathcal J_{\ua}$.
Equivalently,
\[
  \cV_n = \bigl\{\,\ua \in \cU^\infty : ([n]\circ s,e)(\ua)\in\Delta_{\mathcal J/\cU^\infty}\,\bigr\},
\]
where $([n]\circ s,e):\cU^\infty\to \mathcal J\times_{\cU^\infty}\mathcal J$ is the product morphism and
$\Delta_{\mathcal J/\cU^\infty}\subset \mathcal J\times_{\cU^\infty}\mathcal J$ is the relative diagonal. 

Since $\mathcal J\to \cU^\infty$ is an abelian scheme, it is proper and hence separated \cite{Hartshorne1977}*{Section 4}, so its relative diagonal is closed \cite{Hartshorne1977}*{Corollary 4.2}. Taking preimages of closed subschemes under morphisms preserves closedness.

Finally, all maps and
sections are defined over $\Q$, hence so is $\cV_n$.
\end{proof}

\medskip

\paragraph{\emph{The universal section is not torsion.}}

We now prove that $\alpha_{\mathrm{univ}}$ has infinite order in $\mathcal J(\cU^\infty)$. For each of the $16$ possible choices of $h(x)$,
it suffices to find a single specialization $\ua \in \cU^\infty$ such that $\alpha_{\ua}$ is non-torsion. 

Take $h=0$ and consider the curve
\begin{equation}\label{eq:tengely-curve}
  C :\quad Y^2 = X^6 + 18X^5 + 75X^4 + 120X^3 + 120X^2 + 72X + 28.
\end{equation}
By \cite{tengely_thesis}*{Section 4.2}, we know that its Jacobian has rank $1$ and is generated by the divisor
class $[\infty_+-\infty_-]$. This curve corresponds to $\ua = (1,18,75,120,120,72,28) \in \cU^{\infty}(\mathbb Q)$ and 
hence $\alpha_{\ua}$ has infinite order, from which we deduce that $\alpha_{\mathrm{univ}}$ is non-torsion in
$\mathcal J(\cU^\infty)$.

\begin{remark}
For the remaining $15$ choices of $h(x)\in\{0,1\}^4$, we ran an explicit search in the \href{https://www.lmfdb.org/Genus2Curve/Q/}{LMFDB}. For each such $h(x)$ we found at least one specialization
$\ua\in\cU^\infty(\Q)$ for which the associated genus-$2$ curve has Jacobian of analytic rank $1$
and trivial torsion subgroup.
In all cases, the corresponding sextic satisfies the condition
$c=4a_6+h_3^2\in(\Q^\times)^2$, ensuring that the two points at infinity are $\Q$-rational.
For brevity, we do not list these additional examples here but they are recorded in \cite{OnlineComplement}*{remark22\_examples.md file}.
\end{remark}

We now turn the geometric input above into a quantitative statement
for integral models of bounded height. The key point is that, for a curve
$C/\Q$ given by an integral model \eqref{eq:booker-weierstrass}, any torsion specialization
$\alpha_a\in J_a(\Q)$ has order bounded in terms of the height of the model. We need the following technical result.

\begin{proposition} \label{prop:faltings-naive} Let $C$ be a hyperelliptic curve with an integral model given by the equation \eqref{eq:booker-weierstrass} and let $H(C) := \max\limits_{0 \leq i \leq 6} |a_i|$, be the naive height of $C$. Let $\Jac(C)$ be its Jacobian, equipped with its canonical principal polarization, and let $h_F(\Jac(C))$ be its normalized stable Faltings height (as defined in \cite{Pazuki2012}*{Definition 2.1}). Then, for $H(C)$ large enough, we have
$$ h_F(\Jac(C)) \ll \log H(C)$$
where the implied constant is explicit and independent of $C$.

\end{proposition}

\begin{proof}
Kieffer used Thomae's formulae (\cite{Mumford1984}*{IIIa.8.1}) and height bounds for root differences of polynomials (see \cite{Kieffer2022}*{Proposition 5.4}) to prove an explicit upper-bound for the level-4 $\theta$-height (see \cite{Kieffer2022}*{end of proof of Proposition 5.17}). This bound implies that, for $H(C) \geq 2$,
$$ h_{\theta,4}(\Jac(C)) \ll \log H(C),$$  where the implied constant is explicit. Note that Kieffer uses a \textit{naive} height for rational functions, which is  $ \ll \log H(C)$ (see Remark 3 after \cite{Kieffer2022}*{Definition 5.1}).

We next use Pazuki's comparison theorem between $\theta$-height and (stable) Faltings height $h_F$ \cite{Pazuki2012}*{Corollary 1.3(1)} with $g=2$ and $r=4$ which gives

\[
\bigl|h_{\theta,4}(\Jac(C))-\tfrac12 h_F(\Jac(C))\bigr|\ \ll\ \log(h_{\theta,4}(\Jac(C))+2),
\]
with an effective implied constant which is independent of $C$. 

Since $\log(h_{\theta,4}(\Jac(C))+2) \ll \log \log H(C)$ for large $H(C)$, the conclusion follows. 

\end{proof}

Theorem 1.2 of Gaudron--R\'emond \cite{GaudronRemond2025} (with $K=\Q$ and $g=2$) implies that
$|\Jac(C)(\Q)_{tors}|\ll \max\{1,h_F(\Jac(C))\}^2$.
Combining with Proposition~\ref{prop:faltings-naive}, we obtain the following corollary.

\begin{corollary}\label{cor:torsion-logH}
Let $C/\Q$ be a smooth genus $2$ curve with an integral model
given by \eqref{eq:booker-weierstrass}. 
Then for $H(C)$ sufficiently large,
\[
|\Jac(C)(\Q)_{tors}|\ \ll\ (\log H(C))^2,
\]
with an effective absolute implied constant which is independent of $C$.
\end{corollary}

\begin{proposition}\label{prop:torsionlocus}
Using the notations above, define
\[
T(X):=\Bigl\{\ua \in \cU^{\infty}(\mathbb Z): \Ht(\ua)\le X,\ \alpha_{\ua} \in \cJ_{\ua}(\Q)\ \text{is torsion}\Bigr\},
\]
where $\Ht(\ua):=\max_{0\le i\le 6}|a_{i}|$.
Then, for $X$ sufficiently large, one has
\[
\#T(X)\ \ll\ X^{6}(\log X)^{6},
\]
with an effective absolute implied constant independent of $X$. In particular,
\[
\lim_{X \to \infty} \frac{\#T(X)}{\#\{\ua \in \cU^{\infty}(\Z):\Ht(\ua)\le X\}} = 0.
\]
\end{proposition}

\begin{proof}
Let $\ua\in \cU^{\infty}(\Z)$ with $\Ht(\ua)\le X$ and assume that $\alpha_{\ua}$ is torsion.
By Corollary~\ref{cor:torsion-logH} (applied to the integral model parametrized by $\ua$), we have
\[
\bigl|\cJ_{\ua}(\Q)_{\mathrm{tors}}\bigr|\ \ll\ (\log H(C_{\ua}))^{2}\ \ll\ (\log X)^{2},
\]
for $X$ sufficiently large. In particular,
\[
\bigl|\cJ_{\ua}(\Q)_{\mathrm{tors}}\bigr|\ \le\ N(X),
\qquad
N(X):=\bigl\lceil \kappa(\log X)^{2}\bigr\rceil
\]
for some absolute constant $\kappa>0$. Hence there exists an integer $1\le n\le N(X)$ such that
$[n]\alpha_{\ua}=0$, i.e.\ $\ua\in \cV_{n}$. Therefore
\begin{equation}\label{eq:T-in-union-Vn}
T(X) \subseteq\ \bigcup_{n=1}^{N(X)}\Bigl(\cV_{n}(\Z)\cap\{\ua\in\cU^{\infty}(\Z):\Ht(\ua)\le X\}\Bigr).
\end{equation}

For each $n\ge 1$, Proposition~\ref{prop:algebraic-booker} shows that $\cV_{n}$ is Zariski-closed in
$\cU^{\infty}$ (and defined over $\Q$). Moreover, since $\alpha_{\mathrm{univ}}$ is non-torsion in
$\cJ(\cU^{\infty})$, the section $[n]\alpha_{\mathrm{univ}}$ is not identically zero, hence
$\cV_{n}\subsetneq \cU^{\infty}$ is a \emph{proper} closed subset.

We now bound integral points on $\cV_{n}$ uniformly in terms of $n$.
Since $\cU^{\infty}$ is smooth over $\Q$ (hence normal and locally Noetherian), the abelian scheme
$\mathcal J\to \cU^{\infty}$ is projective (see for instance \cite{FaltingsChai}*{Remarks 1.10(a)}) and admits a $\cU^{\infty}$-ample line bundle. Fix such a line bundle $\mathcal L$, replace it by
$\mathcal L\otimes[-1]^{*}\mathcal L$ to assume it is symmetric, and then replace $\mathcal L$ by a fixed
tensor power so that it is relatively very ample and induces a projectively normal embedding
(Zariski-locally on $\cU^{\infty}$) into $\mathbb P^{N}_{\cU^{\infty}}$.
For symmetric $\mathcal L$ one has
\[
[n]^{*}\mathcal L \ \simeq\ \mathcal L^{\otimes n^{2}}
\]
(cf.\ \cite{LombardoNotes}*{Corollary 3.3}); by projective normality it follows that, in these projective
coordinates, the morphism $[n]:\mathcal J\to\mathcal J$ is given by homogeneous polynomials of degree $n^{2}$.
Composing with the section $s:\cU^{\infty}\to\mathcal J$ corresponding to $\alpha_{\mathrm{univ}}$ and comparing
with the zero section, we obtain finitely many polynomial conditions on $\cU^{\infty}$ of total degree $\ll n^{2}$
cutting out $\cV_{n}$. Since $\cV_{n}$ is proper, at least one of these conditions is nontrivial. Recall that, on one fixed \(h\)-component,
\[
   \cU^\infty
   = \operatorname{Spec}\bigl(\mathcal O_{\mathcal U_6}[u]/(u^2-c)\bigr),
   \qquad c=4a_6+h_3^2,
\]
so \(u\) is the coordinate on the quadratic cover
\(\cU^\infty \to \mathcal U_6 \subset \mathbb A^7\). Thus, viewing
$\cU^{\infty}$ as a finite cover of $\cU_{6}\subset \A^{7}$, we may eliminate $u$ and obtain a nonzero polynomial
$F_{n}\in \Z[a_{6},\dots,a_{0}]$ with $\deg(F_{n})\ll n^{2}$ such that every $\ua\in \cV_{n}$ satisfies
$F_{n}(a_{6},\dots,a_{0})=0$.

Let \(S_X=[-X,X]\cap \mathbb Z\). We use the hypersurface estimate
\[
\#\{\mathbf z\in S_X^m:F(\mathbf z)=0\}\le (\deg F)|S_X|^{m-1}
\]
for every nonzero \(F\in \mathbb Z[x_1,\ldots,x_m]\). Indeed, if \(F\) has degree
\(e\) in one variable, then the fibers over points where the leading coefficient is
nonzero contribute at most \(e|S_X|^{m-1}\), while the fibers where the leading
coefficient vanishes are bounded inductively by
\((\deg F-e)|S_X|^{m-1}\). Applying this with \(m=7\) and \(F=F_n\), and using
that \(U^\infty\to U_6\) is given by \(u^2=c\), so each coefficient tuple has at
most two lifts to \(U^\infty(\mathbb Z)\), gives
\[
\#\bigl(V_n(\mathbb Z)\cap\{H\le X\}\bigr)
\le
2\,\#\{\mathbf a\in S_X^7:F_n(\mathbf a)=0\}
\ll (\deg F_n)X^6\ll n^2X^6,
\]
with absolute implied constants independent of \(n\) and \(X\). Combining with \eqref{eq:T-in-union-Vn} yields
\[
\#T(X)\ \ll\ X^{6}\sum_{n=1}^{N(X)} n^{2}\ \ll\ X^{6}\,N(X)^{3}\ \ll\ X^{6}(\log X)^{6}.
\]

Finally, since $\cU^{\infty}\to \cU_{6}\subset \A^{7}$ is finite of degree $2$, we can count that
\[
\#\{\ua \in \cU^{\infty}(\Z):\Ht(\ua)\le X\}\ \asymp\ X^{6}\sqrt{X}\ =\ X^{13/2}.
\]
Hence
\[
\frac{\#T(X)}{\#\{\ua \in \cU^{\infty}(\Z):\Ht(\ua)\le X\}}
\ \ll\ \frac{X^{6}(\log X)^{6}}{X^{13/2}}
\ =\ X^{-1/2}(\log X)^{6}\ \xrightarrow[X\to\infty]{}\ 0,
\]
as claimed.
\end{proof}

As a consequence of Proposition \ref{prop:torsionlocus} we obtain the following immediate corollaries.

\begin{corollary} \label{cor:s1box} The number of pairs $(f,h) \in S_1(X)$ for which $$\left\{ \begin{array}{l} \deg(f) = 6, \\ \Delta(f,h) \neq 0, \\ \mathrm{coeff}_{x^6}(4f(x)+h(x)^2) \text{ is a square, and } \\ \text{the divisor } [\infty_{+} - \infty_{-}] \text{ is torsion on the Jacobian of } y^2+h(x)y=f(x)  \end{array} \right. $$
is of the order $O(X^6(\log X)^6)$, as $X\to \infty$.
\end{corollary}

Let \(S_1^\square(X)\) be the subset of \(S_1(X)\) defined by the first three
conditions in Corollary \ref{cor:s1box}. We claim that
$\#S_1^\square(X)\asymp X^{13/2}$.
Indeed, fix \(h\in\{0,1\}^4\) and write $
c=\operatorname{coeff}_{x^6}(4f+h^2)=4a_6+h_3^2$.
The condition that \(c\) is a square gives \(\asymp X^{1/2}\) possible leading
coefficients \(a_6\) with \(|a_6|\le X\) and \(a_6\ne 0\): if \(h_3=0\), then
\(a_6=t^2\) with \(1\le t^2\le X\), while if \(h_3=1\), then
\(a_6=t(t+1)\) with \(1\le t(t+1)\le X\). For each such \(a_6\), the remaining
six coefficients \(a_5,\ldots,a_0\) have \((2X+1)^6\) possible values before
imposing smoothness. For fixed \(h\) and \(a_6\), the discriminant condition
\(\Delta(f,h)=0\) is a nonzero polynomial condition in \(a_5,\ldots,a_0\): indeed,
the lower coefficients of \(4f+h^2\) are affine coordinates over \(\mathbb Q\),
and there exist separable sextics with any prescribed nonzero leading coefficient.
Thus the elementary hypersurface estimate used above shows that the singular
choices contribute \(O(X^5)\) for each admissible \(a_6\), hence
\(O(X^{11/2})\) choices in total after summing over \(a_6\) and over the finitely
many choices of \(h\). Therefore
\[
\#S_1^\square(X)\asymp X^{1/2}\cdot X^6=X^{13/2}.
\]
Together with Corollary \ref{cor:s1box}, this shows that the proportion of
\((f,h)\in S_1^\square(X)\) for which \([\infty_+-\infty_-]\) is torsion on
\(\operatorname{Jac}(y^2+h(x)y=f(x))\) is
\[
O\bigl(X^{-1/2}(\log X)^6\bigr).
\]
This allows us to prove part (2) of Theorem \ref{th:1}.

\begin{corollary} \label{cor:3} For \((f,h)\in S_1^\square(X)\), write \(J_{(f,h)}\)
for the Jacobian of \(y^2+h(x)y=f(x)\). Then, as \(X\to\infty\),
\[
1-O\bigl(X^{-1/2}(\log X)^6\bigr)
\le
\frac{\#\{(f,h)\in S_1^\square(X): \operatorname{rank} J_{(f,h)}(\mathbb Q)\ge 1\}}
{\#S_1^\square(X)}
\le 1.
\]
In particular,
\[
\liminf_{X\to\infty}
\frac{
\log \left|\{(f,h)\in S_1(X): \operatorname{rank} J_{(f,h)}(\mathbb Q)\ge 1\}\right|
}
{\log |S_1(X)|}
\ge \frac{13}{14}.
\]
\end{corollary}
\begin{proof}
Let \(E_1(X)\subset S_1^\square(X)\) be the subset for which
\(\alpha_{(f,h)}=[\infty_+-\infty_-]\in J_{(f,h)}(\mathbb Q)\) is torsion. By
Corollary \ref{cor:s1box} and the count of \(S_1^\square(X)\) above,
\[
\frac{\#E_1(X)}{\#S_1^\square(X)}
=O\bigl(X^{-1/2}(\log X)^6\bigr).
\]
If \(J_{(f,h)}(\mathbb Q)\) has rank \(0\), then every rational point of
\(J_{(f,h)}(\mathbb Q)\) is torsion; in particular \(\alpha_{(f,h)}\) is torsion.
Thus
\[
S_1^\square(X)\setminus E_1(X)
\subseteq
\{(f,h)\in S_1^\square(X): \operatorname{rank}J_{(f,h)}(\mathbb Q)\ge 1\}.
\]
This gives the displayed lower bound, while the upper bound is trivial. Since
\(\#S_1^\square(X)\asymp X^{13/2}\) and \(\#S_1(X)\asymp X^7\), the claimed logarithmic
density lower bound of \(\frac{13}{14}\) also follows.
\end{proof}

\begin{proposition}\label{prop:rank2-independence-generic}
Assume $h(x)=0$.  Define the subvariety
\[
\cU^{\infty}_{1,1}\ :=\ \bigl\{\ua=(a_6,\dots,a_0)\in \cU^{\infty}\ :\ a_6=1,\ a_0=1\bigr\}.
\]
For $\ua\in \cU^{\infty}_{1,1}$ we write
\[
C_{\ua}:\quad y^2=x^6+a_5x^5+a_4x^4+a_3x^3+a_2x^2+a_1x+1,
\]
and we denote by $\infty_{\pm}$ the two points at infinity.  Let $\cJ\to \cU^{\infty}_{1,1}$
be the restriction of the universal Jacobian, and let
\[
P_{\mathrm{univ}}=(0,-1)\in \cC(\cU^{\infty}_{1,1})
\qquad\text{and}\qquad
\beta_{\mathrm{univ}}:=[P_{\mathrm{univ}}-\infty_{+}]\in \cJ(\cU^{\infty}_{1,1}).
\]
Then the two sections $\alpha_{\mathrm{univ}}|_{\cU^{\infty}_{1,1}}$ and $\beta_{\mathrm{univ}}$
are $\Z$--linearly independent in $\cJ(\cU^{\infty}_{1,1})$.
\end{proposition}

\begin{proof}
Consider the specialization at
\[
C_{0}:\quad y^2=x^6+8x^5+10x^4+10x^3+5x^2+2x+1,
\]
corresponding to $\ua_0=(1,8,10,10,5,2,1)\in \cU^{\infty}_{1,1}(\Z)$.
The LMFDB entry for $C_0$ (see~\href{https://www.lmfdb.org/Genus2Curve/Q/15625/a/15625/1}{LMFDB:15625.a.15625.1}) states that
$\Jac(C_0)(\Q)\simeq \Z^2$ and has trivial torsion, and it lists generators on the (simplified)
model as
\[
G_1=2(0\!:\!-1\!:\!1)-(1\!:\!-1\!:\!0)-(1\!:\!1\!:\!0),
\qquad
G_2=(0\!:\!-1\!:\!1)-(1\!:\!1\!:\!0)
\]
(see \cite{LMFDB}).
With our notation $\infty_{+}=(1\!:\!1\!:\!0)$, $\infty_{-}=(1\!:\!-1\!:\!0)$ and
$P=(0\!:\!-1\!:\!1)$, we have $G_2=\beta_{\ua_0}$ and
\[
G_1-2G_2=(1\!:\!1\!:\!0)-(1\!:\!-1\!:\!0)=\alpha_{\ua_0}.
\]
Hence $\alpha_{\ua_0}$ and $\beta_{\ua_0}$ generate a rank-$2$ subgroup of $\Jac(C_0)(\Q)$,
so in particular they are $\Z$--linearly independent.
If there were a nontrivial relation $m\alpha_{\mathrm{univ}}+n\beta_{\mathrm{univ}}=0$ in
$\cJ(\cU^{\infty}_{1,1})$, it would specialize to the same relation at $\ua_0$, a contradiction.
\end{proof}

On this locus we have that $u^2=4$ and $\cU^\infty_{1,1}$ is the disjoint union
of the two components $u=2$ and $u=-2$. To avoid introducing new notation, for the next two results, we apply the argument on each
component and take the union of the two resulting thin exceptional sets.

\begin{proposition}\label{prop:rank2-thin}
Let $\Gamma\subset \cJ(\cU^{\infty}_{1,1})$ be the subgroup generated by
$\alpha_{\mathrm{univ}}|_{\cU^{\infty}_{1,1}}$ and $\beta_{\mathrm{univ}}$.
There exists a thin subset $\Omega\subset \cU^{\infty}_{1,1}(\Q)$ (in the sense of
\cite{SerreMW}*{\S9.1}) such that for every $\ua\in \cU^{\infty}_{1,1}(\Q)\setminus \Omega$,
the specialization map
\[
\mathrm{sp}_{\ua}:\Gamma\longrightarrow \Jac(C_{\ua})(\Q)
\]
is injective. In particular, for such $\ua$ the points $\alpha_{\ua}$ and $\beta_{\ua}$ are
$\Z$--linearly independent and
\[
\rank\Jac(C_{\ua})(\Q)\ \ge\ 2.
\]
\end{proposition}

\begin{proof}

Since $\cU^{\infty}_{1,1}$ is a Zariski open subset of an affine $5$--space, its function field is
a purely transcendental extension $K=\Q(a_1,\dots,a_5)$.
Let $A/K$ be the generic fiber of the abelian scheme $\cJ\to \cU^{\infty}_{1,1}$.
By Proposition~\ref{prop:rank2-independence-generic}, the subgroup $\Gamma\subset A(K)$ has rank $2$.
N\'eron's specialization theorem (as stated in \cite{SerreMW}*{\S11.1}) asserts that the set of
$\ua\in \cU^{\infty}_{1,1}(\Q)$ for which the specialization homomorphism $A(K)\to \Jac(C_{\ua})(\Q)$
fails to be injective is thin. Taking $\Omega$ to be this thin set gives the claim.
\end{proof}

\begin{corollary}\label{cor:rank2-density-in-box}
Define
\[
\cU_{1,1}^\infty(\mathbb Z;X)
:=
\{a\in \cU_{1,1}^\infty(\mathbb Z): H(a)\le X\},
\qquad
H(a)=\max_{0\le i\le 6}|a_i|.
\]
Then there exists \(\gamma>0\) such that, for \(X\) sufficiently large,
\[
\#\{a\in \cU_{1,1}^\infty(\mathbb Z;X):
\operatorname{rank}\operatorname{Jac}(C_a)(\mathbb Q)<2\}
\ll X^{9/2}(\log X)^\gamma.
\]
Consequently,
\[
1-O\bigl(X^{-1/2}(\log X)^\gamma\bigr)
\le
\frac{
\#\{a\in \cU_{1,1}^\infty(\mathbb Z;X):
\operatorname{rank}\operatorname{Jac}(C_a)(\mathbb Q)\ge 2\}
}
{\#\cU_{1,1}^\infty(\mathbb Z;X)}
\le 1.
\]
\end{corollary}

\begin{proof}
Let $\Omega$ be as in Proposition~\ref{prop:rank2-thin}. For every
\(a\in \cU_{1,1}^\infty(\mathbb Q)\setminus\Omega\), the specialization map is
injective on the subgroup generated by \(\alpha_{\mathrm{univ}}\) and
\(\beta_{\mathrm{univ}}\). Hence \(\alpha_a\) and \(\beta_a\) are
\(\mathbb Z\)-linearly independent in \(\operatorname{Jac}(C_a)(\mathbb Q)\), and
therefore
\[
\operatorname{rank}\operatorname{Jac}(C_a)(\mathbb Q)\ge 2.
\]
Thus
\[
\{a\in \cU_{1,1}^\infty(\mathbb Z;X):
\operatorname{rank}\operatorname{Jac}(C_a)(\mathbb Q)<2\}
\subseteq
\Omega\cap[-X,X]^5.
\]
Since $\Omega$ is thin in $\A^5(\Q)$, Serre's
quantitative estimate for affine thin sets \cite{SerreMW}*{\S13.1, Theorem 1} (with $K=\Q$, $n=5$, $d=1$)
gives
\[
\#\bigl(\Omega\cap[-X,X]^5\bigr)\ \ll\ X^{9/2}(\log X)^{\gamma}.
\]
Moreover $\#\cU^{\infty}_{1,1}(\Z;X)\asymp X^5$ (the complement of $\cU^{\infty}_{1,1}$ is Zariski
closed of codimension $\ge 1$, hence contributes $O(X^4)$ points). The two components $u=\pm2$, mentioned before Proposition \ref{prop:rank2-thin} only change the leading constant. This proves the bound for
the rank smaller than \(2\) locus and hence the displayed lower bound for the proportion of Jacobians of rank greater than or equal to \(2\).
\end{proof}

\begin{corollary}[part (3) of Theorem \ref{th:1}]\label{cor:rank2-logdensity}
One has
\[
\liminf_{X\to\infty}
\frac{\log\bigl|\{(f,h)\in S_1(X): \rank\left(J_{(f,h)}(\Q) \right)\ge 2\}\bigr|}
{\log|S_1(X)|}
\ \ge\ \frac{5}{7}.
\]
\end{corollary}

\begin{proof}
For $X\ge 1$, the family $\cU^{\infty}_{1,1}(\Z;X)$ parametrizes a subset of $\cC_1(X)$ of cardinality
$\asymp X^5$ (five freely varying coefficients, with $a_6=a_0=1$ fixed). By
Corollary~\ref{cor:rank2-density-in-box}, all but $O(X^{9/2}(\log X)^{\gamma})$ of these model curves have
Jacobian rank at least $2$. Since $|S_1(X)|\asymp X^7$, the stated logarithmic density follows.
\end{proof}

Corollaries \ref{cor:3} and \ref{cor:rank2-logdensity} show that, in an arithmetic ordering by height, imposing the condition that the
two points at infinity are $\Q$-rational produces a distinguished rational class
$\alpha_C=[\infty_+-\infty_-]\in \Jac(C)(\Q)$, which is non-torsion for all but a
negligible subset of models.

Moreover, Corollary \ref{cor:rank2-logdensity} gives an unconditional source of \emph{two} independent
rational classes on a large explicit subfamily: on the $5$-parameter family
$U_{\infty}^{1,1}(\Z;X)$ (of logarithmic density $\frac 57$ inside $C_1(X)$).

Completely analogously to $S_1^{\square}(X)$, write $S_2^{\square}(X,Y)$ for the subset of $S_2(X,Y)$ defined by the same conditions. For a fixed $Y \geq 1$, one has $\#S_2^{\square}(X,Y)\asymp_Y X^{13/2}$ when $X \to \infty$. Then, we have the following corollary.

\begin{corollary}\label{prop:box-torsion-booker-S2}
For any fixed positive integer $Y$, when $X \to \infty$ we have
\[
1-O_Y\bigl(X^{-1/2}(\log X)^6\bigr)
\le
\frac{\#\{(f,h)\in S_2^\square(X,Y):
\operatorname{rank}J_{(f,h)}(\mathbb Q)\ge 1\}}
{\#S_2^\square(X,Y)}
\le 1.
\]
\end{corollary}

\begin{remark}\label{rem:higher-genus}
The arguments of this subsection extend to hyperelliptic curves of genus $g\ge 1$:
imposing a condition that guarantees the existence of two points at infinity produces a subfamily
of size $\asymp X^{2g+2}\sqrt{X}$ inside the box of size $\asymp X^{2g+3}$, so one
obtains rank $r\ge 1$ for a set of logarithmic density at least
$\frac{4g+5}{4g+6}$. For rank $r\ge 2$ one would need a second section and to prove
it is generically $\Z$-independent from $[\infty_+-\infty_-]$; without an explicit
specialization witnessing independence, proving this would require other ideas in higher genus.
\end{remark}

\subsection{A lower bound for the number of genus-2 curves with split Jacobian and rank at least 2}\label{ssec:split}

In this subsection we give an explicit construction of {\it many} models for genus-$2$ curves whose Jacobians have Mordell--Weil rank at least $2$. The family we consider here has split Jacobian, so the rank computation reduces to a rank computation on the two elliptic curves factors. We follow the construction of Gajovi\'c--Park in  \cite{Gajovic2025}.

\begin{proposition}\label{prop:split} 
Let $
\begin{array}{l}
\mathcal{N}_{2}(X)=\left\{f \in \cC_1(X) \mid   \,  \rank \Jac\left(C_f\right) \geq 2\right\}
\end{array}
$. Then
$$
|\mathcal{N}_{2}(X)| \gg \frac{X^{2 / 3}}{(\log X)^2}.
$$  
\end{proposition}

We prove the proposition using two auxiliary inputs recorded in Gajovi\'c--Park. For
squarefree integers \(d,m\ne 0\), let
\[
C_{d,m}:\ y^2=d^3x^6+m^3,\qquad
E_d:\ y^2=x^3+d^3,\qquad E_m:\ y^2=x^3+m^3.
\]
The split-Jacobian construction in \cite{Gajovic2025}*{Section 2.2} gives
\(\Jac(C_{d,m})\sim_{\mathbb Q}E_d\times E_m\), hence
\[
\rank\Jac(C_{d,m})(\mathbb Q)=\rank E_d(\mathbb Q)+\rank E_m(\mathbb Q).
\]
We also use the needed case of Frey's rank computation, as cited in
\cite{Gajovic2025}*{proof of Theorem 2.5}: if \(p>3\) is prime,
\(p\equiv3\pmod 4\), then \(\rank E_p(\mathbb Q)=1\).

\begin{proof}[Proof of Proposition \ref{prop:split} ] Fix $X \geq 1$. Consider all pairs of primes $(d, m)$ with

$$
d \leq X^{1 / 3}, \quad m \leq X^{1 / 3}, \quad d \equiv m \equiv 3 \quad(\bmod 4), \quad d, m>3 .
$$

For each such pair, define the polynomial $f_{d, m}(x)=d^3 x^6+m^3$. By construction,  $\left|f_{d, m}\right|=\max \left\{\left|d^3\right|,\left|m^3\right|\right\} \leq X$ so $C_{d,m}$ is a smooth hyperelliptic curve lying in $\cC_1(X)$.

By the cited isogeny and Frey's rank computation, both \(E_d\) and \(E_m\)
have rank \(1\), and hence \(\rank \Jac(C_{d,m})(\mathbb Q)=2\).

Using the Prime Number Theorem for primes in arithmetic progressions, we have
$$\left|\left\{p \leq X^{1 / 3}: p\; \text{prime}, \, p \equiv 3 \bmod 4\right\}\right| \sim \frac{3}{2} \frac{X^{1 / 3}}{\log X}$$ 
hence the number of pairs $(d,m)$ as above is asymptotically $\frac{9}{4} \frac{X^{2 / 3}}{(\log X)^2}$. 
\end{proof}

\subsection{Finding high rank twists of genus-2 curves with split Jacobian}\label{ssec:twist of split curve}

The aim of this section is to explain how to construct twists of genus-$2$ curves with split Jacobians whose Mordell–-Weil rank is high. We will follow the presentation of \cite{HoweLeprevostPoonen2000}, recalling the main constructions and results relevant for our purposes and adapting them to our specific setting.

Let $K$ be a field of characteristic zero, and let $\bar K$ be an algebraic closure of $K$. 
Let $F$ and $G$ be elliptic curves over $K$ given by the equations $y^{2}=f(x)$ and 
$y^{2}=g(x)$, respectively, where $f$ and $g$ are separable monic cubic polynomials in $K[x]$, 
with discriminants $\disc(f)$ and $\disc(g)$. 
Suppose that
\[
\psi \colon F[2](\bar K)\longrightarrow G[2](\bar K)
\]
is an isomorphism of Galois modules that does not arise from an isomorphism 
$F_{\bar K}\to G_{\bar K}$. In the situation of Proposition~\ref{prop:gluing}, since both curves have full $K$-rational $2$-torsion, the isomorphism $\psi$ is given by a bijection between the sets $\{(\alpha_i,0)\}$ and $\{(\beta_i,0)\}$, identifying the $2$-torsion points corresponding to the roots of $f$ and $g$.

Then one can \emph{glue} $F$ and $G$ along $\psi$ as follows. 
Let $H\subset F\times G$ be the graph of $\psi$, viewed as a finite subgroup scheme. 
The quotient $(F\times G)/H$ is an abelian surface defined over $K$, and, as explained below, it is isomorphic to the Jacobian of a genus-$2$ curve defined over $K$.

A precise justification of this construction is given in  \cite{HoweLeprevostPoonen2000}*{Section 3.2}. 
More precisely, the quotient $(F\times G)/H$ carries a natural principal polarization induced by the product polarization on $F\times G$, and it follows from \cite{HoweLeprevostPoonen2000}*{Proposition 3 and Section 3.2} that this principally polarized abelian surface is isomorphic over $K$ to the Jacobian of a genus-$2$ curve.

\begin{proposition}[Prop.~4 in \cite{HoweLeprevostPoonen2000}]\label{prop:gluing}
Let $F\colon y^{2}=f(x)$ and $G\colon y^{2}=g(x)$ be elliptic curves over $K$ with full $K$-rational $2$-torsion.
Let $\alpha_1,\alpha_2,\alpha_3$ (resp. $\beta_1,\beta_2,\beta_3$) be the roots of $f$ (resp. $g$),
indexed cyclically modulo $3$. Define
\[
a_1=\sum_{i=1}^3 \frac{(\alpha_{i+2}-\alpha_{i+1})^2}{\beta_{i+2}-\beta_{i+1}}, \qquad
b_1=\sum_{i=1}^3 \frac{(\beta_{i+2}-\beta_{i+1})^2}{\alpha_{i+2}-\alpha_{i+1}},
\]
\[
a_2=\sum_{i=1}^3 \alpha_i(\beta_{i+2}-\beta_{i+1}), \qquad
b_2=\sum_{i=1}^3 \beta_i(\alpha_{i+2}-\alpha_{i+1}),
\]
and set
\[
A=\disc(g)\frac{a_1}{a_2}, \qquad
B=\disc(f)\frac{b_1}{b_2}.
\]
Define
\[
\glue(f,g)
=
-\prod_{i=1}^3
\bigl(
A(\alpha_{i+1}-\alpha_i)(\alpha_i-\alpha_{i-1})x^2
+
B(\beta_{i+1}-\beta_i)(\beta_i-\beta_{i+2})
\bigr),
\]
where all indices are taken modulo $3$. Then $\glue(f,g)$ is a separable sextic polynomial in $K[x]$.
Let $C\colon y^2=\glue(f,g)(x)$. Let \(H\subset F\times G\) be the graph of the isomorphism
\[
(\alpha_i,0)\longmapsto (\beta_i,0), \qquad i=1,2,3.
\]
Then
\[
\Jac(C)\simeq_K (F\times G)/H.
\]
In particular,
\[
\Jac(C)\sim_K F\times G.
\]
\end{proposition}

For any nonzero integer $d$, let $F^{(d)}$ (resp. $G^{(d)}$) denote the quadratic twist of $F$
(resp. $G$), given by the equation $y^{2}=f^{(d)}(x)$ (resp. $y^{2}=g^{(d)}(x)$), where
$f^{(d)}(x)=d^{3}f(x/d)$ (resp. $g^{(d)}(x)=d^{3}g(x/d)$).

\begin{lemma} \label{gluelemma}
Let $F\colon y^{2}=f(x)$ and $G\colon y^{2}=g(x)$ be elliptic curves over $K$ with full $2$-torsion,
and let $C\colon y^{2}=\glue(f,g)$. For any nonzero integer $d$, the curve
\[
C^{(d)}\colon dy^{2}=\glue(f,g)
\]
is isomorphic to the curve
\[
y^{2}=\glue\bigl(f^{(d)},g^{(d)}\bigr).
\]
\end{lemma}

\begin{proof}
A direct calculation shows that the roots of $f^{(d)}$ (resp. $g^{(d)}$) are $d\alpha_i$
(resp. $d\beta_i$). Moreover,
\[
\disc\bigl(f^{(d)}\bigr)=d^{6}\disc(f), \qquad
\disc\bigl(g^{(d)}\bigr)=d^{6}\disc(g),
\]
and
\[
\glue\bigl(f^{(d)},g^{(d)}\bigr)=d^{21}\glue(f,g).
\]
The change of variables $y\mapsto d^{11}y$ transforms the equation
$y^{2}=\glue\bigl(f^{(d)},g^{(d)}\bigr)$ into $dy^{2}=\glue(f,g)$, as claimed.
\end{proof}

The lemma suggests a manner to find high rank quadratic twists of a genus-$2$
curve \(C\) such that \(\Jac(C)\) is isogenous to \(F\times G\) for two elliptic
curves \(F\) and \(G\) with full rational \(2\)-torsion: one finds integers \(d\)
such that both \(F^{(d)}\) and \(G^{(d)}\) have high rank. Ralph Greenberg proposed a study in terms of Galois representations (see~\cite{Hatley2017} for a review and recent results).

Note that all elliptic curves $E:y^2=f(x)$ with full rational $2$-torsion have a quadratic twist $E^{(d)}:dy^2=f(x)$ which can be put in Legendre form $y^2=x(x-1)(x-\lambda)$, with $\lambda$ in its coefficient field. If $F$ and $G$ are two elliptic curves with full rational $2$-torsion and if there exists a twist $d_0$ such that both $F^{(d_0)}$ and $G^{(d_0)}$ are in Legendre form the following result (see~\cite{Alaa2017}*{Lemma 4.1}) asserts the existence of twists where both curves have positive rank.

\begin{proposition}
Assume the $abc$-conjecture. Let $F\colon y^{2}=f(x)$ and $G\colon y^{2}=g(x)$ be two non-isomorphic elliptic curves over $\Q$ such that
\[
f(x)=x(x-1)(x-\lambda_1), \qquad
g(x)=x(x-1)(x-\lambda_2),
\]
for some distinct $\lambda_1,\lambda_2 \in \Q$.
Let $C$ be the genus-$2$ curve given by the equation
$y^{2}=\glue(f,g).$
For each nonzero integer $d$, consider the quadratic twist $C^{(d)}$ of $C$
given by $dy^{2}=\glue(f,g).$
Then the set of squarefree integers $d$ such that
\[
\rank_{\Q}\Jac(C^{(d)})\ge 2
\]
has logarithmic density at least $\frac{1}{6}$.
\end{proposition}

\begin{proof}
Replacing $\lambda_i$ by an equivalent Legendre parameter if necessary,
we may assume that
\[
0<\lambda_1<\lambda_2<1.
\]
Set
\[
D(u)
=
(\lambda_1-\lambda_2)(u^2-1)
\bigl(1-\lambda_2+(\lambda_1-1)u^2\bigr)
\bigl(\lambda_1u^2-\lambda_2\bigr).
\]
Since the leading coefficient of $D(u)$ is
\[
(\lambda_1-\lambda_2)(\lambda_1-1)\lambda_1>0,
\]
we have $D(u)>0$ for all sufficiently large $u$.
By \cite{Alaa2017}*{Lemma 4.1}, the simultaneous twists
$F^{(D(u))}$ and $G^{(D(u))}$ have positive rank over $\Q(u)$. By the
Néron--Silverman specialization theorem
\cite{SilvermanAdvanced}*{Chapter III, Theorem 11.4}, for all but finitely many
$u_0\in\Q$, the specialized twists $F^{(D(u_0))}$ and $G^{(D(u_0))}$ have positive
rank over $\Q$.

For such a specialization, let $d$ be the squarefree representative of the class of
$D(u_0)$ in $\Q^\times/(\Q^\times)^2$. Then $F^{(D(u_0))}\simeq_\Q F^{(d)}$ and
$G^{(D(u_0))}\simeq_\Q G^{(d)}$. Moreover, by Lemma~\ref{gluelemma} and
Proposition~\ref{prop:gluing},
\[
\Jac(C^{(d)})\sim_\Q F^{(d)}\times G^{(d)}.
\]
Hence, since Mordell--Weil rank is invariant under isogeny,
\[
\rank_\Q\Jac(C^{(d)})
=
\rank_\Q F^{(d)}+\rank_\Q G^{(d)}
\ge 2.
\]

Since $D(u)$ is a squarefree polynomial of degree $6$, Theorem~3.5 of
\cite{Poonen2003}, whose proof relies on the $abc$-conjecture, implies that the
number of squarefree representatives in $\Q^\times/(\Q^\times)^2$ represented by the values
$D(u)$ with $1\le u\le B$ is $\gg B$.
Choose $a>0$ such that $D(u)\le X$ whenever
\[
1\le u\le aX^{1/6}
\]
and $X$ is sufficiently large. For each such value of $u$, let $d$ be the positive
squarefree representative of the class of $D(u)$ in $\Q^\times/(\Q^\times)^2$.
Then there exists a constant $C>0$, depending only on $\lambda_1$ and $\lambda_2$,
such that
\[
d\le C D(u).
\]
Hence $d\le CX$, and therefore the number of positive squarefree integers
$d\le X$ obtained in this way is $\gg X^{1/6}$.
Removing the finitely many exceptional specializations does not affect this lower bound.

Since $|D(X)|\asymp X$, it follows that
\[
\liminf_{X\to\infty}
\frac{
\log\bigl|\{d\in D(X):\rank_\Q\Jac(C^{(d)})\ge 2\}\bigr|
}{
\log |D(X)|
}
\ge \frac{1}{6}.
\]
This proves the claimed logarithmic density lower bound.
\end{proof}

We next describe the distribution of ranks in families of quadratic twists
of genus-$2$ Jacobians that are isogenous to the square of an elliptic curve admitting
a rational $3$-isogeny.

We fix a hyperelliptic model $C:y^2=h(x)$ over $K$. For $d\in K^\times$, we denote by
$C^{(d)}:dy^2=h(x)$ the quadratic twist relative to this model. Over $K(\sqrt d)$, the curves $C$ and $C^{(d)}$ are isomorphic via
$(x,y)\mapsto (x,\sqrt d\,y)$. The associated descent cocycle is given by the
hyperelliptic involution $(x,y)\mapsto(x,-y)$, which induces multiplication by
$-1$ on $\Jac(C)$. It follows that $\Jac(C^{(d)})$ is the quadratic twist of
$\Jac(C)$ by $[-1]$.

\begin{proposition}\label{square twist}
Let \(E/\mathbb Q\) be a semistable elliptic curve admitting a
\(\mathbb Q\)-rational \(3\)-isogeny, and let \(C/\mathbb Q\) be a genus-\(2\)
curve such that $
\Jac(C)\sim_{\mathbb Q} E^2$ .
With \(D(X)\) as in \eqref{eq:twist-DX}, one has
\[
\liminf_{X\to\infty}
\frac{
\#\{d\in D(X): \rank_{\mathbb Q}\Jac(C^{(d)})=2\}
}{
|D(X)|
}
>0 .
\]
\end{proposition}

\begin{proof}
Let $
\mathcal D(X)=
\{\Delta>0:\Delta<X,\ \Delta\text{ is a fundamental discriminant}\}$.
Since \(E\) is semistable and admits a \(\mathbb Q\)-rational \(3\)-isogeny,
\cite{KrizLi2019}*{Proposition 9.7} gives
\[
\#\{\Delta\in\mathcal D(X):
\ord_{s=1}L(E^{(\Delta)},s)=1\}\gg_E X .
\]
From the works of Gross--Zagier \cite{GrossZagier1986} and Kolyvagin \cite{Kol89} it follows that analytic rank \(1\) implies
Mordell--Weil rank \(1\),  hence
$
\#\{\Delta\in\mathcal D(X):
\rank_{\mathbb Q}E^{(\Delta)}=1\}\gg_E X $.

We now pass from positive fundamental discriminants to positive squarefree twist
parameters. For \(\Delta\in\mathcal D(X)\), define
\[
\operatorname{sf}(\Delta)=
\begin{cases}
\Delta, & \Delta\equiv 1 \pmod 4,\\
\Delta/4, & \Delta\equiv 0 \pmod 4.
\end{cases}
\]
Then \(\operatorname{sf}(\Delta)\in D(X)\), the map
\(\Delta\mapsto \operatorname{sf}(\Delta)\) is injective, and
$
\Delta/\operatorname{sf}(\Delta)\in\mathbb Q^{\times 2}$.
Therefore
$
E^{(\Delta)}\simeq_{\mathbb Q}E^{(\operatorname{sf}(\Delta))}$.
Moreover, it is known  (see for instance \cite{FouvryKlueners2007}*{page 467}) that
$
|D(X)|\sim \frac{6}{\pi^2}X$ and
$|\mathcal D(X)|\sim \frac{3}{\pi^2}X$,
so \(|\mathcal D(X)|/|D(X)|\to 1/2\). It follows that
\[
\#\{d\in D(X):\rank_{\mathbb Q}E^{(d)}=1\}\gg_E X,
\]
and hence this set has positive lower proportion inside \(D(X)\). For each such \(d\),
\[
\rank_{\mathbb Q}\Jac(C^{(d)})
=
2\,\rank_{\mathbb Q}E^{(d)}
=
2,
\]
and the conclusion follows.
\end{proof}

The next result describes the distribution of the Mordell–Weil ranks of Jacobians in the family of double quadratic twists of genus-$2$ curves with split Jacobian.

\begin{proposition} \label{split twist}

Let \(E_i:y^2=f_i(x)\), \(i=1,2\), be elliptic curves over
\(\mathbb Q\), where \(f_i\) are monic separable cubic polynomials split over
\(\mathbb Q\). Choose orderings of the roots of \(f_1\) and \(f_2\), and assume
that the resulting identification \(E_1[2]\simeq E_2[2]\) does not arise from an isomorphism $E_{1, \overline{\mathbb Q}} \simeq E_{2, \overline{\mathbb Q}}$. Let
\(C:y^2=\operatorname{glue}(f_1,f_2)\) be the corresponding genus-\(2\) curve.
Assume that the Birch and Swinnerton-Dyer conjecture holds for the quadratic
twist families of \(E_1\) and \(E_2\).

For nonzero squarefree integers \(d_1,d_2\), let
\(E_i^{(d_i)}:y^2=f_i^{(d_i)}(x)\), where
\(f_i^{(d_i)}(x)=d_i^3 f_i(x/d_i)\), and let
\(C^{(d_1,d_2)}\) be the genus-\(2\) curve obtained by applying the same gluing
construction to \(E_1^{(d_1)}\) and \(E_2^{(d_2)}\).

With \(D(X)\) as in \eqref{eq:twist-DX}, one has
\[
\lim_{X\to\infty}
\frac{\bigl|\{(d_1,d_2)\in D(X)^2 :
\rank\Jac(C^{(d_1,d_2)})=r\}\bigr|}
{|D(X)|^2}
=
\begin{cases}
\frac{1}{4}, & r=0,\\[4pt]
\frac{1}{2}, & r=1,\\[4pt]
\frac{1}{4}, & r=2,\\[4pt]
0,           & r\ge 3.
\end{cases}
\]
\end{proposition}

\begin{proof}
By construction, $E_1^{(d_1)}$ (resp.~$E_2^{(d_2)}$) is isomorphic to
$E_1$ (resp.~$E_2$) over $\Q(\sqrt{d_1})$ (resp.~$\Q(\sqrt{d_2})$). Over \(L=\mathbb Q(\sqrt{d_1},\sqrt{d_2})\), the maps
\[
(x,y)\mapsto (x/d_i,\; y/d_i^{3/2})
\]
identify \(E_i^{(d_i)}\) with \(E_i\) and carry the ordered set of nontrivial
\(2\)-torsion points of \(E_i^{(d_i)}\) to that of \(E_i\). Hence they identify
the graph used in the gluing of \(E_1^{(d_1)}\) and \(E_2^{(d_2)}\) with the
graph used in the gluing of \(E_1\) and \(E_2\). Therefore the corresponding
principally polarized quotients are isomorphic over \(L\), and  \(C^{(d_1,d_2)}\) and \(C\) are isomorphic over \(L\).

 It follows that
\[
\Jac(C^{(d_1,d_2)}) \sim_{\Q} E_1^{(d_1)} \times E_2^{(d_2)}.
\]
By \cite{Smith2025}*{Corollary 1.2}, under BSD, the quadratic twists of \(E_i\) have rank
\(0\) and \(1\) with proportions \(1/2\) each, while the twists of rank at least
\(2\) have proportion \(0\). Since
\[
\rank \Jac(C^{(d_1,d_2)}) =
\rank E_1^{(d_1)} + \rank E_2^{(d_2)},
\]
the claimed distribution of ranks follows.
\end{proof}

\section{Cryptographic motivation: Regev's quantum algorithm}
\label{sec:cryptomotivation}

In this section we develop the cryptographic discussion from
Section~\ref{ssec:crypto} and recall the features of Regev's DLP algorithm that
are used below. Let \((G,+)\) be the group in which the DLP is posed, for us an
elliptic curve group or the Jacobian of a genus-\(2\) curve over \(\F_q\), and
let \(h=[x]g\) be the target relation. Let \(n\) be the binary size of \(G\).
The step that dominates the gate complexity of Regev's algorithm requires
\(d+\frac{n}{d}\) additions and doublings in \(G\), while the gate complexity per
run of Shor's algorithm is \(n\) additions and doublings. Hence the asymptotic
advantage of Regev's algorithm with respect to Shor's algorithm is
\(\min(d,\frac{n}{d})\). When \(G\) is the multiplicative group of a finite
field, one can take \(d=\sqrt{n}(\log n)^{O(1)}\) (see~\cite{Pilatte2026}). For
elliptic curves and genus-\(2\) Jacobians, the largest value currently obtained
for \(d\) is \((\log n)^{\frac{1}{2}}\)
(see~\cite{BarbulescuBarcauPasol2025}). In curve applications, increasing \(d\)
amounts to finding suitable rational Mordell--Weil generators, as recalled next.

\subsection{Regev's algorithm in a nutshell}\label{ssec:Regev}

For the chosen parameter \(d\), the algorithm uses \(d-2\) auxiliary elements of
\(G\), denoted \(g_1,\ldots,g_{d-2}\), and sets \(g_{d-1}=h\) and \(g_d=g\).
It considers the lattice
 \begin{equation}\label{eq:L}
 L=\left\{(z_1,\ldots,z_d)\in \Z^d\mid \sum_{i=1}^d [z_i]g_i=0\right\}. 
 \end{equation}
One computes in superposition all the sums $\sum_{i=1}^d [z_i]g_i$, called multi-scalar product, where $z_1,\ldots,z_d$ are non-negative $\frac{n}{d}$-bit integers. Next, a quantum procedure computes a basis of the lattice~$L$. Finally, one solves a linear system to find a vector of $L$ of the form $(0,\ldots,0,-1,*)$. The last coordinate of this vector is $\log_gh$. 

To obtain the stated complexity, one needs a classical algorithm that computes
arbitrary multi-scalar products with coefficients of $\frac{n}{d}$ bits. This can
be done using Pippenger's algorithm~\cite{Pippenger1980}, provided one can store
a table of all multi-scalar products with coefficients in $\{0,1\}$. In the
quantum setting such large lookup tables are not available. To overcome this
obstacle, Regev~\cite{Regev2023} chooses the auxiliary elements
$g_1,\ldots,g_{d-2}$ so that they can be represented using few bits. 

As anticipated in Section~\ref{ssec:crypto}, the parameter $d$ is supplied, for
curves, by reducing Mordell--Weil generators from a rational lift. Let
$C_0/\F_q$ be the finite-field curve on which the DLP is posed, and let
$\widetilde C/\Q$ be a curve with good reduction at $q$ whose reduction is
isomorphic to $C_0$. This is the setting used
in~\cite{BarbulescuBarcauPasol2025}*{Section~3.3}
and~\cite{BarbulescuBisson2024}*{Section~5.1}. If
$J=\Jac(\widetilde C)$, good reduction gives a specialization homomorphism
$
\rho_q:J(\Q)\longrightarrow \Jac(C_0)(\F_q)
$. 
Choose points $P_1,\ldots,P_r\in J(\Q)$ which are independent modulo torsion,
and write $\overline P_i=\rho_q(P_i)$. These reduced points are used as the
auxiliary elements in the lattice construction, namely
\[
g_i=\overline P_i\quad (1\le i\le r), \quad g_{r+1}=h,\quad g_{r+2}=g,
\]
where $h$ is the DLP input and $g$ is the base point. Thus $d-2=r$, so a
rank-$r$ lift allows one to take
\[
d=r+2,
\]
provided that the reduced points satisfy the usual independence
heuristic which makes the lattice $L$ in~\eqref{eq:L} behave as in Assumption 1
of \cite{Ekera2024}. This heuristic fails in degenerate cases such as the one
described in~\cite{Ekera2024}*{Example preceding Assumption~1}, while the
numerical experiments in~\cite{BarbulescuBarcauPasol2025}*{Section~3.2} support
it when the chosen auxiliary elements behave independently.

Solving Problem~\ref{prob:draw} produces curves on which Regev's algorithm can
be run with a large parameter \(d\) (see Appendix~\ref{appendix}); these are
useful for experiments and comparisons with Shor's algorithm, but they do not by
themselves give a curve compatible with a fixed HECC instance.

For a fixed curve \(C/\F_q\), Problem~\ref{prob:twist} asks for a high-rank
rational lift, twist, or related curve whose reduction is compatible with \(C\).
If the twisting parameter \(d\), or the two parameters \((d_1,d_2)\), are squares
modulo \(q\), then the corresponding twist has reduction isomorphic to \(C\).
More generally, if \(\psi\) is an isogeny, in particular an isomorphism, between
two algebraic varieties of order \(|G|\), with kernel of order coprime to
\(|G|\), then
$$\log_gh=\log_{\psi(g)}\psi(h).$$
Thus the DLP can be transported through this map. In this sense, solving
Problem~\ref{prob:twist} for a rational lift \(\widetilde C\) of \(C\) supplies
a compatible curve on which Regev's algorithm can be run with a larger parameter.

\subsection{High rank twists of certain curves used in cryptography}\label{ssec:twist of split}

Problem~\ref{prob:twist} is solved on a list of HECC genus $1$ (resp. genus $2$) curves  in~\cite{BarbulescuBarcauPasol2025} (resp.~\cite{BarbulescuBisson2024}). The empirical observation which motivated this work is that a naive enumeration offers better results in the case of two types of genus-$2$ curves: CM curves of simple Jacobian and the curves which are twists of a split Jacobian. In the rest of this section we obtain a justification for this phenomenon.

\subsubsection{\emph{The curves $C_a:y^2=x^5+a$.}}\label{sssec:x5}
The Buhler--Koblitz curve 4GLV127-BK  (see~\cite{BarbulescuBisson2024}) has equation $y^2=x^5+17$
 and is a twist of $y^2+y=x^5$ (see~\cite{Buhler1998}). The curve is one of the most important in HECC, e.g. it was considered along the curve \emph{Generic1271} in~\cite{Bos2014}. It has CM and simple Jacobian (see \href{https://www.lmfdb.org/Genus2Curve/Q/3125/a/3125/1}{LMFDB:4096.b.65536.1}). One of its twists, the curve $y^2=x^5+8$, is used in~\cite{Brown2005}.
 
 Problem~\ref{prob:twist} was solved for $y^2=x^5+17$
in~\cite{BarbulescuBisson2024}: an early-abort search enumerated the twists
$y^2=x^5+17\delta$ for $1\leq \delta\leq 128\cdot 10^4$ in
$448\,960$ core-seconds, about $125$ core-hours, and found twists up to
rank~$7$. For the compatible finite-field instance used there, a rank-$6$
twist gives Regev parameter $8$ and an eightfold speedup over Shor; thus the
search is worthwhile as offline preprocessing in that example.

 In an experiment, we used a modified search which uses the particular properties of the simple CM curves. Among the  curves $y^2=x^5+a$ with $a\leq 10^5$ one has:
 \begin{itemize}
 \item $60794$ are squarefree (this is close to $\zeta(2)^{-1}$, the theoretical proportion of squarefree positive integers, see \cite{Graham1981} or \cite{FouvryKlueners2007}*{page 467});
  \item $3559$ of them have at least $3$ rational points of height $\leq 500$ on $C_a$ (this is implemented in Magma, see~\cite{Stoll2016}). Heuristically, when embedded in the Jacobian $J_a=\Jac(C_a)$ they are not torsion with high probability and we expect $\rank J_a\geq 3$. 
  \item $\geq1663$ of them have a root number of opposite parity with respect to the lower bound obtained above, so the rank is larger than the previous bound. The root number can be computed by a closed formula proven in~\cite{Bisatt2022}.
 \item For these curves one computes in Magma a basis of the Mordell--Weil
group. This is implemented in Magma~\cite{magma} V2.27-6 and
PARI/GP~\cite{parigp} version 2.18 only for genus $1$ and~$2$. We find
$15,11,2,2$ twists of ranks $5,6,7,8$, respectively. If compatible with a
finite-field target, a rank-$8$ example would give Regev parameter $10$; this
enumeration was not recorded as a timed benchmark. For comparison, the
similarly sized data set of~\cite{Booker2016} contains no genus-$2$ curve of
rank $\geq 5$.
  \end{itemize}
 
 \begin{remark}
 As an alternative to computing a basis of the Mordell--Weil group, one can compute the analytic rank. The method is interesting because the curves $C_a$ are CM and in the genus $1$ analogue, twists of high rank of the CM curve $y^2=x^3-x$ have been found by computing the analytic rank (see e.g.~\cite{Elkies2002}).    
 
 To extend their idea to genus 2, note that Stoll~\cite{Stoll2002} identified the Hecke character $\eta_a$ such that $L(J_a)=L(\eta_a)$ (see also~\cite{Stoll2003}). The computation of the Hecke grossen characters has an implementation in PARI/GP since version 2.15 using the algorithms in~\cite{Molin2022}. We are indebted to Aurel Page for the program that we used to compute the L-function of $J_a$ (see~\cite{OnlineComplement}). However, with the current implementation, the computation of the L-function of $J_a$ using the Hecke character is more than $20$ times slower than the general method to compute L-functions of genus 2 curves, which is implemented in PARI/GP since version 2.10. It is approximately $70$ times slower than the optimized implementation of Bill Allombert. Note however that, contrary to the elliptic curves analogue, in the case of genus 2 curves the Hecke grossen characters have relatively new implementations and the numerical results in this work were obtained without this alternative.
\end{remark}

\subsubsection{Curves whose Jacobians become isogenous to \(E\times E\).}

Freeman and Satoh defined two families of curves $C$ whose Jacobian is the square of an elliptic curve:
\begin{lemma}[Prop. 4.1 and 4.2 in \cite{Freeman2011}]\label{lem:FreemanSatoh}
Let $K$ be a perfect field and let $C$ be a hyperelliptic curve defined over $K$ which is defined by Equation~\eqref{eq:Cdeg5} (resp.~\eqref{eq:Cdeg6}).
\begin{eqnarray}
C:y^2&=&x^5+ax^3+bx \label{eq:Cdeg5}\\
C:y^2&=&x^6+ax^3+b.\label{eq:Cdeg6}
\end{eqnarray}
Let $m=4$ (resp. $m=3$).
 Set $c=a/\sqrt{b}\in \overline{K}$ and $\zeta_m$ a primitive $m$-th root of unity. Then $\Jac(C)$ is isogenous over $K(b^{\frac{1}{2m}},\zeta_m)$ to $E\times E$ where $E$ is defined by Equation~\eqref{eq:Edeg5} (resp.~\eqref{eq:Edeg6}):
\begin{eqnarray}
E(c):y^2&=&(c+2)x^3-(3c-10)x^2+(3c-10)x-(c+2)\label{eq:Edeg5}\\
E(c):y^2&=&(c+2)x^3-(3c-30)x^2+(3c+30)x-(c-2)   .\label{eq:Edeg6}
\end{eqnarray}
\end{lemma}

\medskip

\paragraph{\emph{The curves $y^2=x^5+bx$.}}~

They occur in cryptography e.g.  $b=21$ in~\cite{Drylo2012}*{Ex 20} and $b=3$ in~\cite{Kawazoe2008}. Note that the curves are twists over $\Q(b^{\frac{1}{5}})$ of $y^2=x^5+x$, which is of the form~\eqref{eq:Cdeg5} with $a=0$ and $b=1$. By Lemma~\ref{lem:FreemanSatoh} we have
$$\Jac(C_1)\sim_{\mathbb Q(i)} E\times E 
\text{ with }
E:y^2= x^3+10x^2-20x-8.
$$
(see \href{https://www.lmfdb.org/EllipticCurve/Q/256/a/2}{LMFDB:256.a2}). We found the twist $d=110814$ of rank~$4$.\medskip

\paragraph{\emph{Drylo's pairing-friendly curves with
split Jacobian}}~ 
Drylo~\cite{Drylo2012} proposed a list of curves as in the statement of Lemma~\ref{lem:FreemanSatoh} with $K=\Q$. A priori, the parameter $b$ is not required to be a square, but the second part of Example~19 in that article is a rational square. 

Indeed, Drylo considered the $1/10$-twist of $C:y^2=x^6+ax^3+b$ with $a=4/25$ and $b=(8/125)^2$; here $c=\frac{a}{\sqrt{b}}\in \Q$. By Lemma~\ref{lem:FreemanSatoh}, $\Jac(C)\sim_{\Q(\zeta_3)} E\times E$ where $E=E(\frac{5}{2})$ as in Equation~\eqref{eq:Cdeg6}. The quadratic twist of $d=1046$ of the curve $y^2 = x^3 + \frac{5}{2}x^2 - \frac{45}{4}x - \frac{81}{4}$ has rank $4$, so the corresponding twist of $C$ has rank $8$ over $\Q(\zeta_3)$. 

Thus, in both split examples, a rank-$4$ elliptic twist gives genus-$2$
rank $8$ over the splitting field, hence Regev parameter $10$ and
dominant-operation gain $n/(10+n/10)$; these twists are examples rather than
timed search benchmarks.

\appendix

\section{List of hyperelliptic curves suited for an implementation of Regev's algorithm}\label{appendix}

An implementation of Regev's algorithm for hyperelliptic curves benefits from
explicit curves for which the parameter $d$ can be taken as large as possible.
We collect a short list of small-coefficient, high-rank examples for
Problem~\ref{prob:draw}; they are not evidence for Theorem~\ref{th:1}, but
inputs for experiments with the Regev-style algorithm discussed in
Section~\ref{sec:cryptomotivation}.

\subsection{Genus-2 curves with simple Jacobian and high rank}\label{ssec:simple}
\emph{A curve of small height}.
Booker et al.~\cite{Booker2016} found $66158$ $\Q$-isomorphic classes of curves 
of absolute discriminant less than $10^6$.
They are available in the \href{https://www.lmfdb.org/Genus2Curve/Q/}{LMFDB}.
The largest rank in that list is $4$, e.g. the curve
\[
y^2+(x^3+x+1)y=x^5-x^4-5x^3+9x+6
\]
(see \href{https://www.lmfdb.org/Genus2Curve/Q/440509/a/440509/1}{LMFDB:440509.a.440509.1})
has a Jacobian of rank $4$. Booker and Sutherland~\cite{Sutherland2026} have
announced work in progress on a larger genus-$2$ database.

\medskip
\paragraph{\emph{Curves of high rank}} We summarize a list of record curves in the literature. Here $\widehat{h}$ denotes the base-$2$ logarithm of the maximum canonical height among the basis elements.
\begin{center}
\begingroup
\scriptsize
\setlength{\tabcolsep}{3pt}
\renewcommand{\arraystretch}{1.12}
\begin{tabular}{@{}p{0.70\textwidth}ccc@{}}
\hline
curve & ref. & rank & $\widehat{h}$\\
\hline
\(\begin{aligned}[t]
y^2={}&1306881x^6 + 18610236x^5 - 46135758x^4\\
&{}-1536521592x^3 -2095359287x^2 + 32447351356x + 89852477764
\end{aligned}\)
& \cite{Stahlke1997} & $\geq16$ & 8.41\\
\hline
\(\begin{aligned}[t]
y^2={}&82342800x^6 - 470135160x^5 + 52485681x^4 + 2396040466x^3\\
&{}+ 567207969x^2 - 985905640x + 247747600
\end{aligned}\)
& \cite{Stoll2016} & $\geq22$ & 14.73\\
\hline
\(\begin{aligned}[t]
y^2={}&4037229x^6 + 34187102x^5 - 724533076x^4 - 4944866082x^3\\
&{}+57659086152x^2 + 241518308040x + 313383220164
\end{aligned}\)
& \cite{Dreier1997} & $\geq25$ & 28.55\\
\hline
\(\begin{aligned}[t]
y^2={}&80878009x^6 - 236558406x^5 - 1018244179x^4\\
&{}+4436648480x^3 + 6445563464x^2 - 13620761544x + 68406^2
\end{aligned}\)
& \cite{Elkies2008} & $\geq26$ & 32.14\\
\hline
\(\begin{aligned}[t]
y^2={}&60516x^6 + 1680225324x^5 + 200663873413x^4 - 1021197439562x^3\\
&{}-290505889943111x^2 + 316071770416320x + 123355813282790400
\end{aligned}\)
& \cite{Elkies2015} & $\geq27$ & 39.57\\
\hline
\end{tabular}
\endgroup
\end{center}

\subsection{Highest known ranks for genus-2 curves with split Jacobian}\label{ssec:highest}

By Proposition~\ref{prop:gluing}, gluing elliptic curves with full rational
$2$-torsion gives split genus-$2$ Jacobians, and for
$J=\Jac(C)\sim_K E_1\times E_2$ the Mordell--Weil rank is
$\rank E_1(K)+\rank E_2(K)$; in the diagonal case this is $2\rank E$. Thus the
congruent-number curve \(E:y^2=x^3-x\)
(\href{https://www.lmfdb.org/EllipticCurve/Q/32/a/3}{LMFDB:32.a3}) and
\(C:y^2=6x^6-9x^4-9x^2+6\), with \(\Jac(C)\sim_{\Q}E\times E\), give
genus-\(2\) twists of rank \(2s\) from rank-\(s\) twists of \(E\); see
Rogers~\cite{Rogers2000} and Watkins~\cite{Watkins2014}*{Table 2} for
\(s=2,\ldots,7\). Over \(\Q(t)\), Shioda~\cite{Shioda1997} found a rank-\(14\)
split example, and combining Elkies's
\(E_1(\Q)\simeq(\Z/2\Z)^2\times\Z^{15}\) curve~\cite{Elkies2009} with the
rank-\(5\) full-\(2\)-torsion curve from
\cite{DujellaPeral2020}*{Example 19} gives a rank-\(20\) genus-\(2\) Jacobian
\((2,2)\)-isogenous to \(E_1\times E_2\); the resulting sextic model has
coefficients in \(\Q(t)\) of degree \(756\) and binary size \(1761\).

\bibliographystyle{alpha}
\bibliography{references}

\end{document}